\documentclass[12pt]{article}
\usepackage{amsmath}
\usepackage{epsfig}
\usepackage{amsthm,amsfonts}
\newtheorem{theorem}{Theorem}[section]
\newtheorem{prop}[theorem]{Proposition}
\newtheorem{cor}[theorem]{Corollary}
\newtheorem{lemma}[theorem]{Lemma}

\newtheorem{remark}[theorem]{Remark}
\newtheorem{definition}[theorem]{Definition}
\newtheorem{notation}[theorem]{Notation}
\begin{document}
\title{The space of Schwarz--Klein spherical triangles}
\author{Alexandre Eremenko\thanks{Both authors are supported by NSF grant DMS-1665115.}~
and Andrei Gabrielov${}^*$}
\maketitle
\def\Z{\mathbb{Z}}
\def\R{\mathbb{R}}
\def\N{\mathbb{N}}
\def\K{\mathbf{K}}
\def\A{\mathbf{A}}
\def\B{\mathbf{B}}
\def\C{\mathbf{C}}
\def\D{\mathbf{D}}
\def\G{\mathbf{G}}
\def\S{\mathcal{S}}
\def\U{\mathcal{U}}
\def\V{\mathcal{V}}
\def\L{\mathcal{L}}
\def\X{\mathcal{X}}
\def\bC{\overline{\mathbb C}}
\begin{abstract}
We describe the space of spherical triangles
(in the sense of Schwarz and Klein). It is a smooth connected orientable
$3$ manifold, homotopy equivalent to the $1$-skeleton of the cubic partition
of the closed first octant in $\R^3$. The angles and sides are
real analytic functions on this manifold which embed
it to $\R^6$.

\end{abstract}
\section{Introduction}

The word ``triangle'' can have two different meanings:
a curve (Ptolemy, Gauss, M\"obius, Study) or a surface (Schwarz, Klein).
In this paper we use it in the sense of Schwarz and Klein. The references
for the space of Gauss-Study triangles are \cite{Chisholm,Study}.

Spherical triangles with arbitrarily large interior angles and sides occur in
the study of spherical metrics with conic singularities which recently
attracted substantial attention, \cite{E,EGT1,EGT2,EMP,MP1,FKKRUY}.
Here we describe topology of the space of all such triangles.

A spherical triangle $\Delta$
is a closed oriented disk with
three distinct marked boundary points ({\em corners})
$a_1,a_2,a_3$ equipped with
a Riemannian metric of constant curvature $1$ with conic singularities
with interior angles $\pi\alpha_i$ at $a_i$ and such that the {\em sides}
$(a_i,a_{i+1})$ are geodesic.
Two such triangles $\Delta$ and $\Delta'$ are considered equal if
there is an orientation-preserving isometry $\Delta\to \Delta'$
mapping corners to corners and preserving the labels.

We always assume that the cyclic order $(a_1,a_2,a_3)$ is consistent
with the positive orientation of $\partial\Delta$.

A {\em developing map} $f:\Delta\to\bC$ is
associated with each triangle $\Delta$.
Here $\bC$ is the Riemann sphere with the standard Riemannian metric
of curvature $1$. The {\em side lengths} of $\Delta$ are the lengths of the
images of the sides $f([a_i,a_{i+1}])$,
counting multiplicity. So the angles and side lengths of $\Delta$
are positive numbers.

A spherical triangle is completely determined by these six positive numbers,
the angles and side lengths. This defines the topology, induced from $\R^6$, on the set
of all spherical triangles.

We measure the angles in half-turns instead of radians, thus ``an integer
angle'' means an integer multiple of $\pi$ radians.
Similarly, the side lengths are measured so that the length of a great circle is 2.
We use notation $(A,B,C)$ for the interior angles of a triangle at its corners $(\A,\B,\C)=(a_1,a_2,a_3)$,
and $(a,b,c)$ for the lengths of its sides opposite the corners with angles $(A,B,C)$.

The following characterization of angles of
spherical triangles was proved in~\cite{E}.
\vspace{.1in}

\noindent{\bf Theorem A.}
{\rm\bf(i)} \emph{If three positive numbers $(A,B,C)$ are not integers, then they
are the angles of a spherical triangle if and only if}
\begin{equation}\label{cond1}
\cos^2\pi A+\cos^2\pi B+\cos^2\pi C+2\cos\pi A\cos\pi B\cos\pi C-1<0.
\end{equation}
\emph{A triangle with such angles is unique.}
\vspace{.1in}
\noindent
{\rm\bf(ii)} \emph{If exactly one of the angles $(A,B,C)$, say $A$, is an integer
then a triangle with these angles exists if and only if either $B+C$ or $|B-C|$
is an integer $m$ of the opposite parity to $A$, and}
\begin{equation}\label{b1}
m\leq A-1.
\end{equation}
\emph{For any such angles, a one-parametric family of triangles with these
angles exists. The length of the side opposite the integer corner is an integer,
 while the length of the side opposite to one of the non-integer corners is not an integer,
 and its fractional part can be chosen as a parameter.}\footnote{In \cite{E} it was erroneously stated that in this case the triangle with given angles is unique.
The mistake was corrected in \cite{FKKRUY}.}
\vskip .1in
\noindent
{\rm\bf(iii)} \emph{If two angles are integers then all three are integers, and
in this case a triangle with the angles $(A,B,C)$ exists if and only if
$A+B+C$ is odd, and}
\begin{equation}\label{b2}
\max(A,B,C)\leq (A+B+C-1)/2.
\end{equation}
\emph{There is a two-parametric family of triangles with such angles:
any two side lengths can be chosen as parameters.
The ranges for the side lengths depend on the angles and will be specified in Section \ref{neighbor} below.}
\vskip.1in

\begin{remark}\label{rmk:A}
\emph{Inequality (\ref{cond1}) is equivalent to}
\begin{equation}\label{cond2}
\cos\pi\,\frac{A\!+\!B\!+\!C}{2}\,\cos\pi\,\frac{A\!+\!B\!-\!C}{2}\,\cos\pi\,\frac{A\!+\!C\!-\!B}{2}
\,\cos\pi\,\frac{\!B\!+\!C\!-\!A}{2}<0.
\end{equation}
\emph{Yet another illuminating form of this inequality was discovered in
\cite{MP1}:}
\begin{equation}\label{cond3}
d_1\left((A,B,C),\Z_e^3\right)>1,
\end{equation}
\emph{where $\Z_e^3$ is the sublattice of $\Z^3$ consisting of points
with even sum of coordinates, and $d_1\big((A,B,C),(M,N,K)\big)=|A-M|+|B-N|+|C-K|$.}
\end{remark}

We use Theorem A to describe the space $\X\subset\R^6$ of all spherical triangles:

\begin{theorem}\label{main} The space $\X$ of all spherical triangles
is a smooth, connected, orientable real analytic $3$-manifold,
homotopy equivalent to the 1-skeleton of a
cubic partition of the closed first octant of $\R^3$.

The angles and side lengths of triangles, which define the embedding of $\X$ to $\R^6$, are smooth (real analytic) functions on $\X$.
\end{theorem}

The plan of the paper is as follows.
In Section \ref{angles} we describe the set $X\in\R^3$
of all possible
triples $(A,B,C)$
of angles of spherical triangles in $\X$. The set $X$ is the union of open tetrahedra $T_{m,n,k}$, where $(m,n,k)\in\N^3$,
some of their open edges, and some of their vertices. In Section \ref{edgesandvertices} we describe the edges and vertices of the tetrahedra
$T_{m,n,k}$ belonging to $X$. An important difference appears between {\em balanced} (satisfying the triangle inequality (\ref{balanced}))
and {\em unbalanced} triangles. This terminology was proposed in \cite{EMP}, see also \cite{EGMP}. In particular, only balanced vertices of the tetrahedra $T_{m,n,k}$ belong to $X$.
In Section \ref{neighbor} we define the neighborhoods $\U_{m,n,k}\in\X$, homeomorphic to open balls in $\R^3$, of the sets of triangles in $\X$ corresponding to the balanced vertices $(m,n,k)\in X$.
The union $\U$ of these neighborhoods is the set of all {\em short-sided} triangles (with all three sides shorter than the full circle, see Definition \ref{short}). All balanced triangles are short-sided.
The covering of $\U$ by the charts $\U_{m,n,k}$ defines the structure of an oriented three-dimensional manifold on $\U$.
In Section \ref{semibalanced} we define the charts $\S_{m,n,k}\subset\X$, where either $m-1=n+k$ or $n-1=m+k$ or $k-1=m+n$, corresponding to {\em semi-balanced} edges in $X$ (see Definition \ref{balancededges}). Together with the charts $\U_{m,n,k}$, this defines a covering of the whole set $\X$,
defining the structure of an oriented three-dimensional manifold on $\X$. In Section \ref{homotopy} we show that the set $\X$ is homotopy equivalent to
the 1-skeleton of a cubic partition of the closed first octant of $\R^3$. The same is true for the set $\U\subset\X$ of short-sided triangles, and for the set $\mathcal B\subset\U$ of balanced triangles.
\vspace{.1in}

We thank Dmitri Panov for discussions of spherical triangles;
he explained
to us their relevance for description of spherical metrics on tori,
and stressed the
importance of the balanced condition (\ref{balanced}).

\section{Angles of spherical triangles}\label{angles}
According to Theorem A, the set $X$ of possible triples $(A,B,C)$ of the angles of spherical triangles consists of three parts: the open set $X_3$ described in (i),
the one-dimensional set $X_1$ described in (ii) and the set of isolated
points $X_0$ described in (iii). We have $X_3=Y_3\cap\R_{>0}^3$ where $Y_3$ is invariant under the group $G$ (see \cite{E}) consisting of transformations
$$(A,B,C)\mapsto (\pm A+M,\pm B+N,\pm C+K),\;\mbox{where}\;(M,N,K)\in\Z^3_e.$$
The interior of a fundamental region of $G$ can be described by the inequalities
\begin{equation}\label{in1}
A>0,\; B>0,\; C>0,\; A+B<1,\; A+C<1,\; B+C<1.
\end{equation}
The intersection of $X_3$ with the region (\ref{in1}) is described by the additional inequality
\begin{equation}\label{in2}
A+B+C>1.
\end{equation}
Conditions (\ref{in1}) and (\ref{in2}) define an open tetrahedron $\nabla$
with vertices $(1,0,0)$, $(0,1,0)$, $(0,0,1)$, $(\frac12,\frac12,\frac12)$.
Three transformations of the group $G$,
\begin{equation}\label{g1}
\begin{split}
(A,B,C)\mapsto(A,1-B,1-C),\\(A,B,C)\mapsto(1-A,B,1-C),\\(A,B,C)\mapsto(1-A,1-B,C),
\end{split}
\end{equation}
map the tetrahedron $\nabla$ onto three disjoint
tetrahedra, each having a common facet with $\nabla$.
The union of these four tetrahedra, with their common facets, the vertex $(\frac12,\frac12,\frac12)$,
and four edges adjacent to that vertex, is an open tetrahedron $T_0$ with the vertices  $(1,0,0)$, $(0,1,0)$, $(0,0,1)$, $(1,1,1)$, which is called a truncated cube in \cite{MP1}.
The subgroup $G_0$ of $G$ preserving $T_0$ consists of the identity and transformations (\ref{g1}). It is the Klein Viergroup isomorphic to $\Z_2\times \Z_2$.
The group $G_1$ consisting of transformations $(A,B,C)\mapsto \pm(A,B,C)+(M,N,K)$, where $(M,N,K)\in\Z^3_e$, is a normal subgroup of $G$ such that $G/G_1=G_0$.
The interior of a fundamental region of $G_1$ is the open unit cube $0<A<1,\;0<B<1,\;0<C<1$.
Thus $G_1$ maps $T_0$ to one tetrahedron $T_{m,n,k}$ in each unit cube
$$Q_{m,n,k}=\{ m< A< m+1,\; n< B< n+1,\; k< C< k+1,\quad (m,n,k)\in\Z^3\}.$$
The group $G_2$ of translations $(A,B,C)\mapsto(A+M,B+N,C+K)$ with $(M,N,K)\in\Z^3_e$ is a normal subgroup of $G_1$, and $G_1/G_2$ is generated by any involution $(A,B,C)\mapsto(M-A,N-B,K-C)$, where $(M,N,K)\in\Z^3_e$, mapping $T_{0,0,0}$ to $T_{M-1,N-1,K-1}$.
Vertices of the tetrahedra $T_{m,n,k}$ belong to the set $\Z^3_o\subset\Z^3$ of integer points with odd sum of coordinates.
The set $X_3$ is the union of the tetrahedra $T_{m,n,k}$ with $(m,n,k)\in\N^3$.
Two tetrahedra in adjacent unit cubes have a common edge (which may belong or not belong to $X_1$).
Since three of the four vertices of $T_0$ have the same sum of angles $A+B+C=1$,
three vertices of each tetrahedron $T_{m,n,k}$ have the same sum, either $m+n+k+1$ or $m+n+k+2$ when $m+n+k$ is even or odd.
It follows from (\ref{cond3}) that the points $(A,B,C)$ of open facets of any tetrahedron $T_{m,n,k}$ never correspond to spherical triangles,
since the angles $A,B,C$ are non-integer and $d_1\big((A,B,C),\Z_e^3\big)=1$ at these points.

\begin{notation}\label{denote}\emph{
We denote by $\V_{m,n,k}\subset\X$ the set of triangles with the angles $(m,n,k)\in\Z^3_o$ corresponding to a vertex $(m,n,k)\in X_0$.
If $L\subset X_1$ is an open edge of a tetrahedron $T_{m,n,k}$, we denote by $\L_{u,v,w}\subset\X$ the set of triangles corresponding to $L$, where $(u,v,w)$ is the midpoint of $L$.
For example, $\L_{1,\frac12,\frac12}$ is the set of triangles with the angles $A=1,\;0<B<1,\;0<C<1$ corresponding to the edge $L=\big((1,0,0),(1,1,1)\big)$ of $T_0$.}
\end{notation}

\section{Edges and vertices of the tetrahedra $T_{m,n,k}$}\label{edgesandvertices}

To see which edges and vertices of the tetrahedra $T_{m,n,k}$ correspond to spherical triangles, we use inequalities (\ref{b1}) and (\ref{b2}).
If all three angles $A,B,C$ are integer then $A+B+C$ is odd, and inequality (\ref{b2}) implies
the triangle inequality
\begin{equation}\label{balanced}
A\le B+C,\quad B\le A+C,\quad C\le A+B.
\end{equation}
Inequalities (\ref{balanced}) define a closed cone $\K$ in the first octant bounded by the planes $A=B+C,\; B=A+C,\; C=A+B$.
The edges of $\K$ are three rays $\{A=B,\;C=0\}$, $\{A=C,\;B=0\}$, $\{B=C,\;A=0\}$.
Klein calls spherical triangles with the angles satisfying (\ref{balanced}) ``triangles of the first kind.'' We prefer to call them {\em balanced} following Mondello and Panov.

\begin{definition}\label{balancededges}{\rm
A point $(A,B,C)$ in the first octant satisfying the inequality (\ref{balanced}) is called \emph{balanced}, otherwise it is \emph{unbalanced}.
A tetrahedron $T_{m,n,k}$ is \emph{balanced} if all its vertices are balanced, \emph{unbalanced} if all its vertices are unbalanced, and \emph{semi-balanced} otherwise. An open edge $L\in X_1$ of a tetrahedron $T_{m,n,k}$ is \emph{balanced} if both its ends are balanced vertices, \emph{unbalanced} if both its ends are unbalanced vertices, and \emph{semi-balanced} otherwise.
An unbalanced end $(m,n,k)\in\Z_o$ of a semi-balanced edge is called a \emph{marginally unbalanced} vertex.
Its angles satisfy either $m-1=n+k$ or $n-1=m+k$ or $k-1=m+n$.}
\end{definition}

\begin{remark}\label{rmk:balanced}{\rm
Theorem A (iii) says that $X_0=\K\cap\Z^3_o$, thus $X_0$ is the set of all balanced vertices.
Note that balanced vertices do not belong to coordinate planes, since (\ref{balanced}) implies $B=C$ when $A=0$, thus $(A,B,C)\notin\Z_o$.
All triangles corresponding to a balanced edge are balanced, all triangles corresponding to an unbalanced edge are unbalanced.
Triangles corresponding to the points of a semi-balanced edge $L$ between its balanced vertex and midpoint (including the midpoint) are balanced, triangles corresponding to the points of $L$ beyond its midpoint are unbalanced.
It is shown below that a semi-balanced tetrahedron in $X_3$ may be either {\em pointed}, with three edges in $X_1$ meeting at a vertex $V\in X_0$, or {\em not pointed}, with two opposite edges in $X_1$. All balanced tetrahedra are pointed, and all unbalanced tetrahedra are not pointed.
The set of balanced triangles in a semi-balanced tetrahedron is described in Proposition \ref{balanced-mnk} below.}
\end{remark}

To determine which edges of the tetrahedra $T_{m,n,k}$ belong to $X_1$,
we consider the edges of eight tetrahedra with a common vertex $(m,n,k)\in\Z_o^3\cap\R^3_{>0}$.
The edges meeting at $(m,n,k)$ are of two types: those on which the sum $A+B+C$ is constant
({\em first type}) and those on which it is not ({\em second type}).
An edge $L$ of the second type is {\em upward} with respect to the vertex $(m,n,k)$ if $A+B+C>m+n+k$ on $L$
and {\em downward} otherwise.
It follows from Theorem A (ii) that
all edges of the first type belonging to $X_1$ are unbalanced.
\vspace{.1in}

{\bf Edges adjacent to a balanced vertex.} Let $V=(m,n,k)\in X_0$ be a balanced vertex. Then edges of the first type adjacent to $V$
do not belong to $X_1$, while six edges of the second type belong to $X_1$: three upward edges from $V$ to
$(m,n+1,k+1)$, $(m+1,n,k+1)$ and $(m+1,n+1,k)$,
and three downward edges from $V$ to $(m,n-1,k-1)$, $(m-1,n,k-1)$ and $(m-1,n-1,k)$.
There are four upward (with $A+B+C>m+n+k$) and four downward (with $A+B+C<m+n+k$) tetrahedra adjacent to $V$.
This is summarized in the following statement.

\begin{prop}\label{sixedges} Six edges of the second type adjacent to a balanced vertex $V=(m,n,k)$ (three upward and three downward) belong to $X_1$.
All upward edges adjacent to $V$ are balanced.
Tetrahedra $T_{m,n,k}$ and $T_{m-1,n-1,k-1}$ adjacent to $V$ are pointed. Each of them has three edges in $X_1$ adjacent to $V$. Each of the remaining six tetrahedra adjacent to $V$ has exactly one edge $(V,V')$ in $X_1$ adjacent to $V$, common either with $T_{m,n,k}$ or with $T_{m-1,n-1,k-1}$. If the edge $(V,V')$ is balanced, the tetrahedron is pointed, having two more edges in $X_1$ adjacent to $V'$. Otherwise, the tetrahedron is not pointed and has two opposite edges in $X_1$.
\end{prop}

{\bf Edges adjacent to an unbalanced vertex.} Let $V=(m,n,k)\in\Z^3_o\cap\R^3_{>0}$ be an unbalanced vertex such that $m+n\leq k-1$.
Then $V\notin X_0$, and the edges meeting at $V$ belong to $X_1$ if and only if $C=k$ on those edges:
the edges of the first type from $V$ to $(m-1,n+1,k)$ and to $(m+1,n-1,k)$,
and the edges of the second type from $V$ to $(m+1,n+1,k)$ and to $(m-1,n-1,k)$.
Therefore exactly four of the twelve edges meeting at $V$ belong to $X_1$.
Each of these four edges is common to two tetrahedra adjacent to $V$, so there are four pairs of these tetrahedra.
Two of these pairs, $(T_{m-1,n,k-1},T_{m-1,n,k})$ and $(T_{m,n-1,k-1},T_{m,n-1,k})$, have common edges of the first type (see Fig.~\ref{edges-type1} for $m+n<k-1$ and Fig.~\ref{semibalanced-type1} for $m+n=k-1$). Two other pairs, $(T_{m-1,n-1,k-1},T_{m-1,n-1,k})$ and $(T_{m,n,k-1},T_{m,n,k})$, have common edges of the second type (see Fig.~\ref{edges-type2} and Fig.~\ref{semibalanced-type2}).
If $m+n<k-1$ then all eight tetrahedra are unbalanced, and each of them has two opposite edges in $X_1$.
If $m+n=k-1$ then vertices $(m+1,n+1,k)$, $(m,n+1,k-1)$ and $(m+1,n,k-1)$, shown as black dots in Fig.~\ref{semibalanced-type1} and Fig.~\ref{semibalanced-type2}, are balanced.
All other vertices of the tetrahedra adjacent to $V$ are unbalanced.
The tetrahedron $T_{m,n,k-1}$ in Fig.~\ref{semibalanced-type2} is pointed, with three edges in $X_1$ meeting at its vertex $(m+1,n+1,k)$.
Each of the other seven tetrahedra in Fig.~\ref{semibalanced-type1} and Fig.~\ref{semibalanced-type2} has two opposite edges in $X_1$.
This is summarized in the following statement.

\begin{prop}\label{fouredges} If $V=(m,n,k)\in\Z^3_o\cap\R^3_{>0}$ is an unbalanced vertex then four edges adjacent to $V$, the angle $\max(m,n,k)$ being constant on those edges, belong to $X_1$.
Two of these edges are of the second type (one upward and one downward) and the other two are of the first type.
If $V$ is marginally unbalanced then there are four semi-balanced and four unbalanced tetrahedra adjacent to $V$.
Two of the semi-balanced tetrahedra, one of them pointed and another one not pointed, have a common semi-balanced upward edge in $X_1$ adjacent to $V$.
The other two semi-balanced tetrahedra are not pointed, have no common edges, and their edges adjacent to $V$ are unbalanced, of the first type.
If $V$ is not marginally unbalanced, all eight tetrahedra adjacent to $V$ are unbalanced.
\end{prop}

\begin{remark}\label{edges-unbalanced}{\em Note that every unbalanced edge in $X_1$ is an edge of at least one unbalanced tetrahedron.
Each of the unbalanced edges in Fig.~\ref{edges-type1} - Fig.~\ref{semibalanced-type2} belongs to an unbalanced tetrahedron shown in those figures,
except the edge $\big((m,n+1,k+1),(m+1,n,k+1)\big)$ in Fig.~\ref{semibalanced-type2} which is common to a semi-balanced tetrahedron
$T_{m,n,k}$ and an unbalanced tetrahedron $T_{m,n,k+1}$ with the vertices $(m,n,k+2)$ and $(m+1,n+1,k+2)$ not shown in Fig.~\ref{semibalanced-type2}.}
\end{remark}

\section{Neighborhoods of balanced vertices}\label{neighbor}
\begin{figure}
\centering
\includegraphics[width=4.5in]{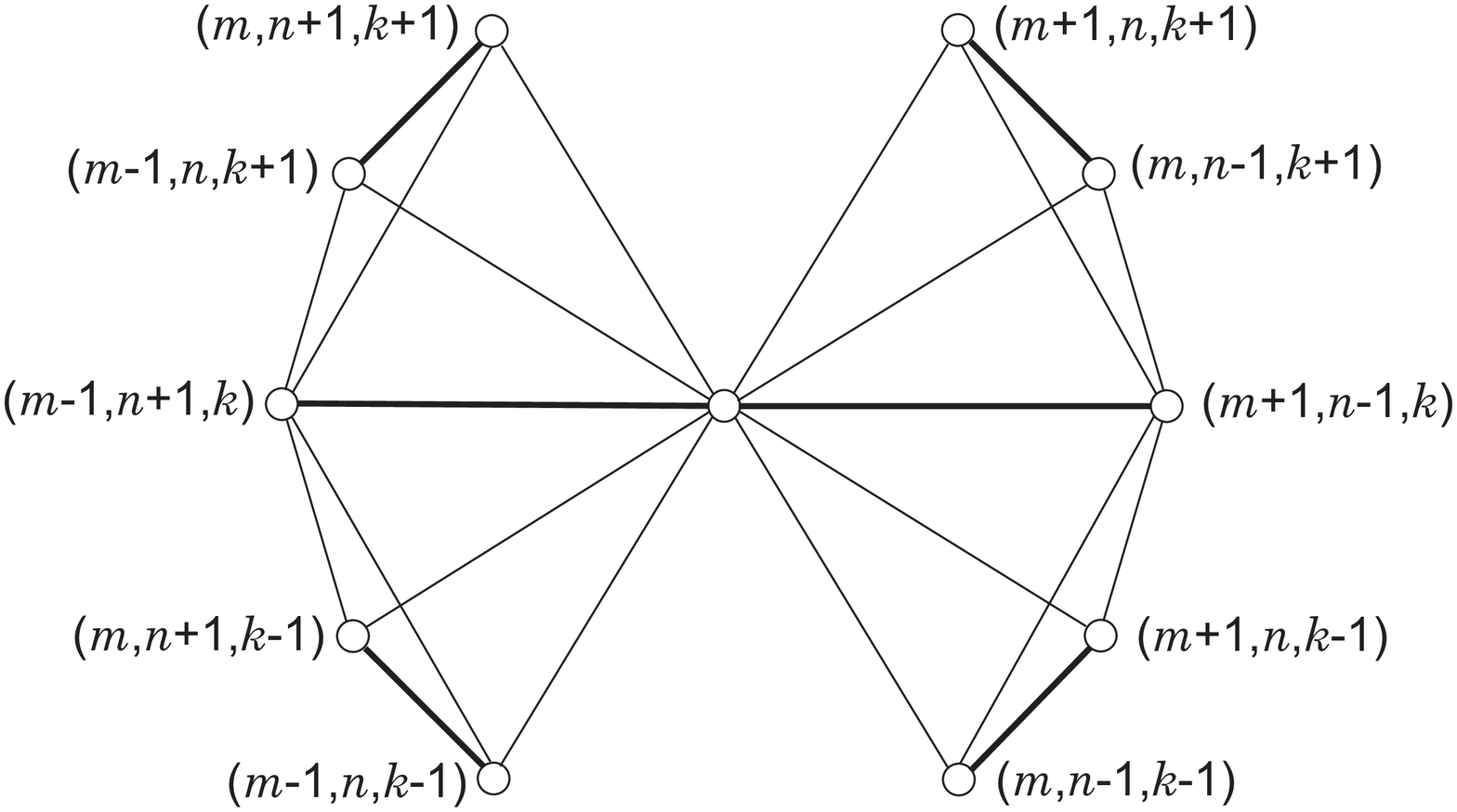}
\caption{Tetrahedra adjacent to edges in $X_1$ of the first type meeting at a vertex $(m,n,k)$ with $m+n<k-1$. Edges in $X_1$ are shown in bold line.}\label{edges-type1}

\bigskip
\includegraphics[width=4.5in]{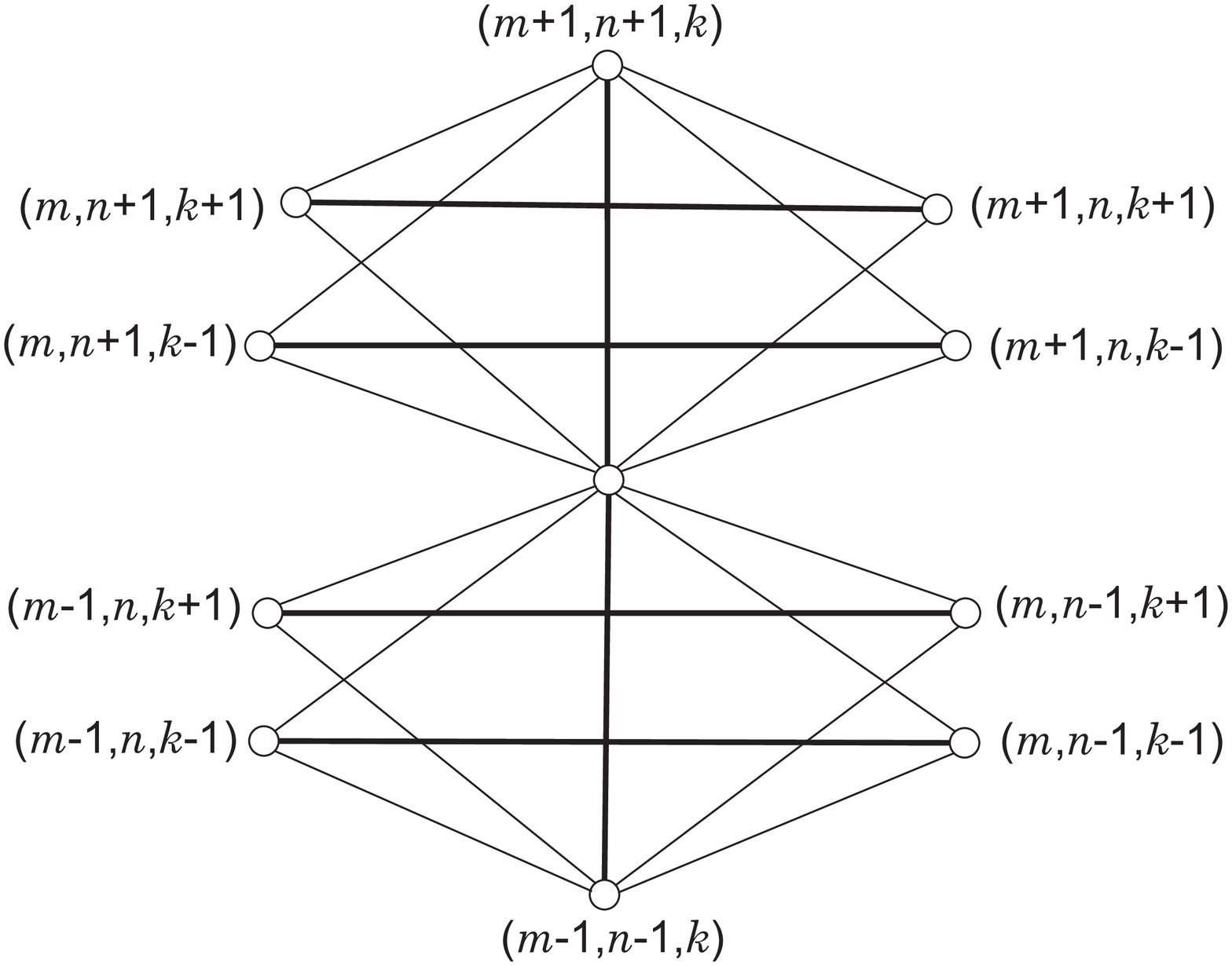}
\caption{Tetrahedra adjacent to edges in $X_1$ of the second type meeting at a vertex $(m,n,k)$ with $m+n<k-1$. Edges in $X_1$ are shown in bold line.}\label{edges-type2}
\end{figure}

\begin{figure}
\centering
\includegraphics[width=4.4in]{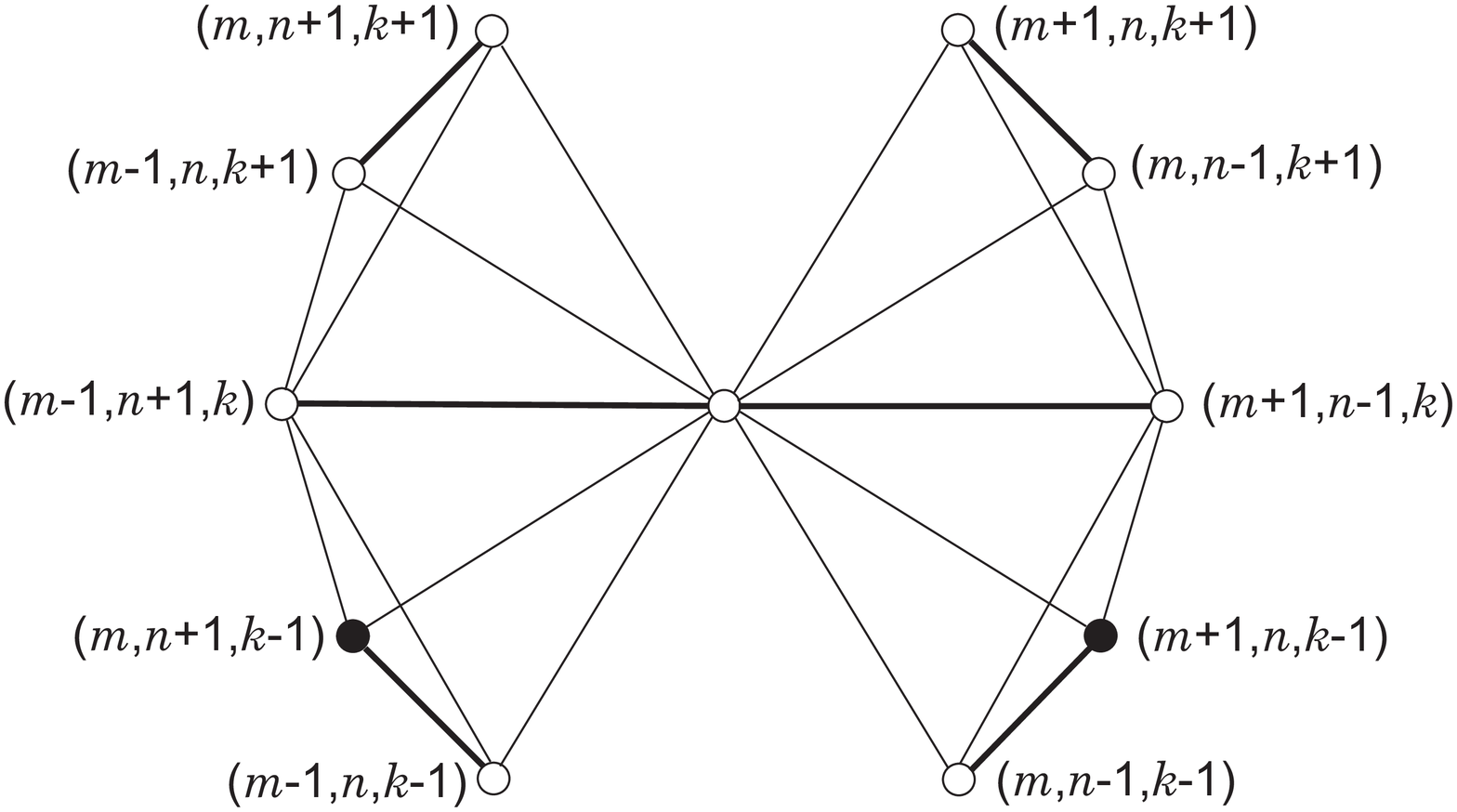}
\caption{Tetrahedra adjacent to edges in $X_1$ of the first type meeting at a marginally unbalanced vertex $(m,n,k)$ with $m+n=k-1$. Balanced vertices are shown as black dots.}\label{semibalanced-type1}

\bigskip
\includegraphics[width=4.4in]{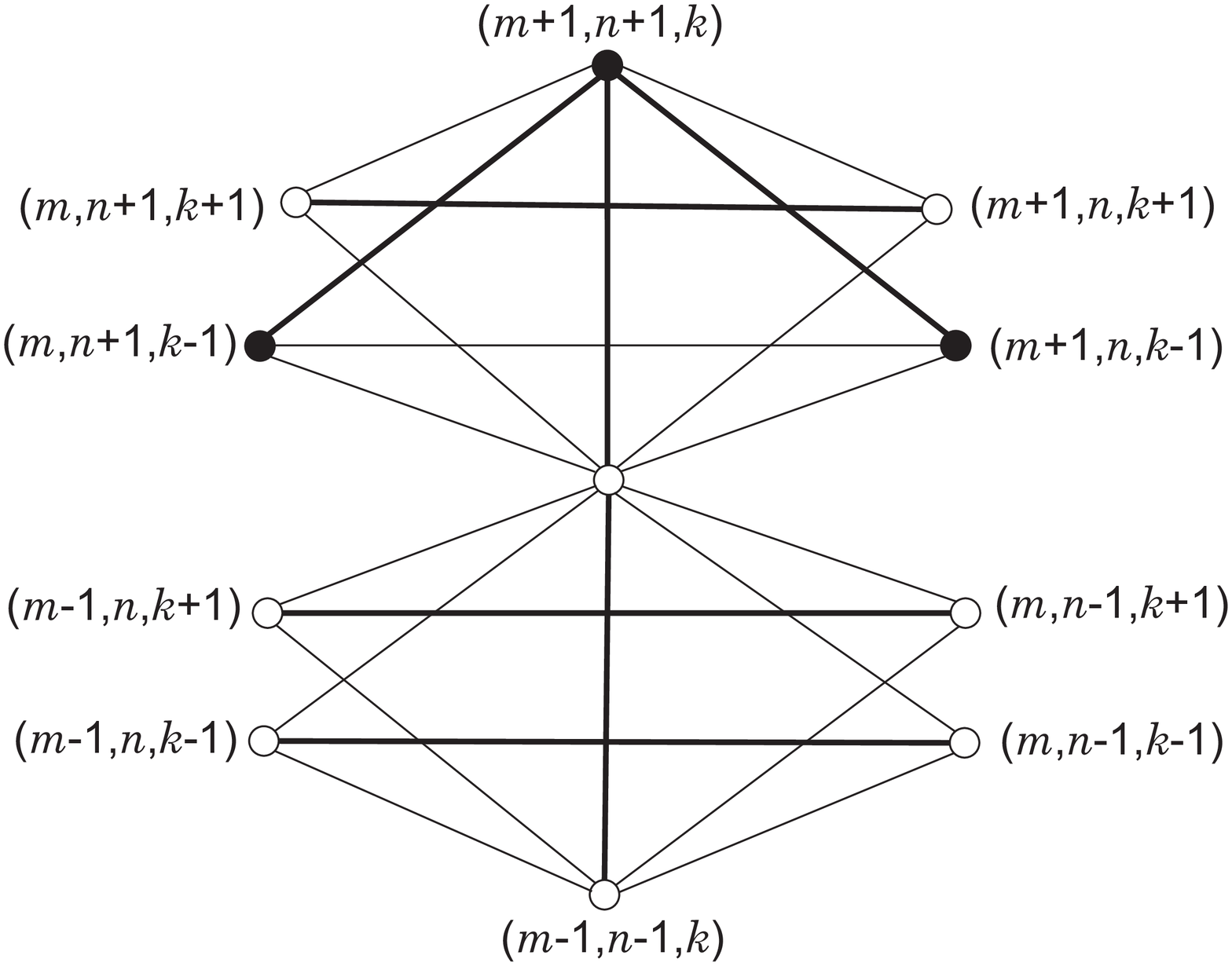}
\caption{Tetrahedra adjacent to edges in $X_1$ of the second type meeting at a marginally unbalanced vertex $(m,n,k)$ with $m+n=k-1$. Balanced vertices are shown as black dots.}\label{semibalanced-type2}
\end{figure}

\begin{figure}
\centering
\includegraphics[width=5.4in]{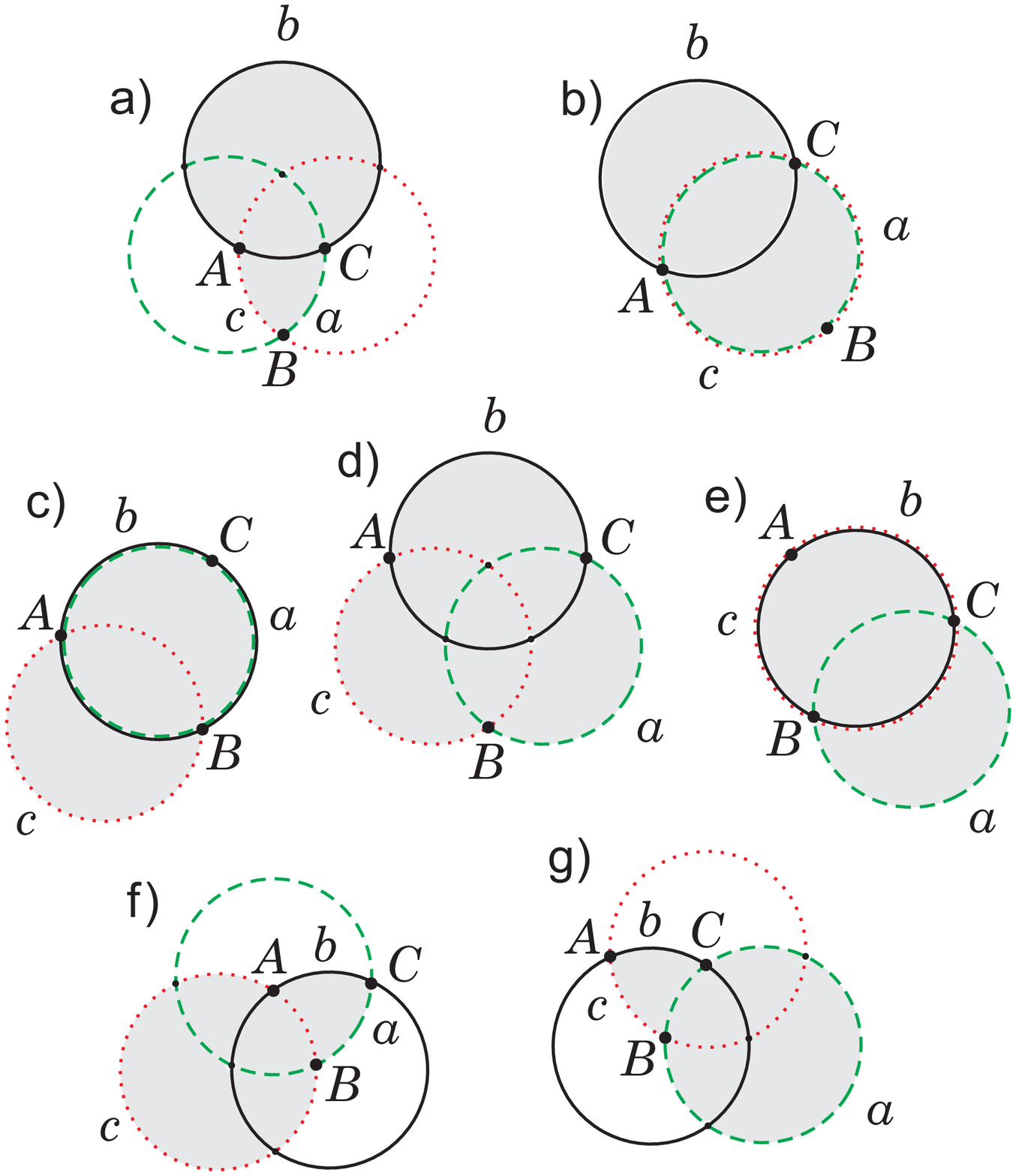}
\caption{Triangles with $A+B+C>3$ in a neighborhood $U$ of the vertex $(1,1,1)$.
When two sides of a triangle are mapped to the same circle, that circle is shown in dual color/style.}\label{up}
\end{figure}
\begin{figure}
\centering
\includegraphics[width=5.4in]{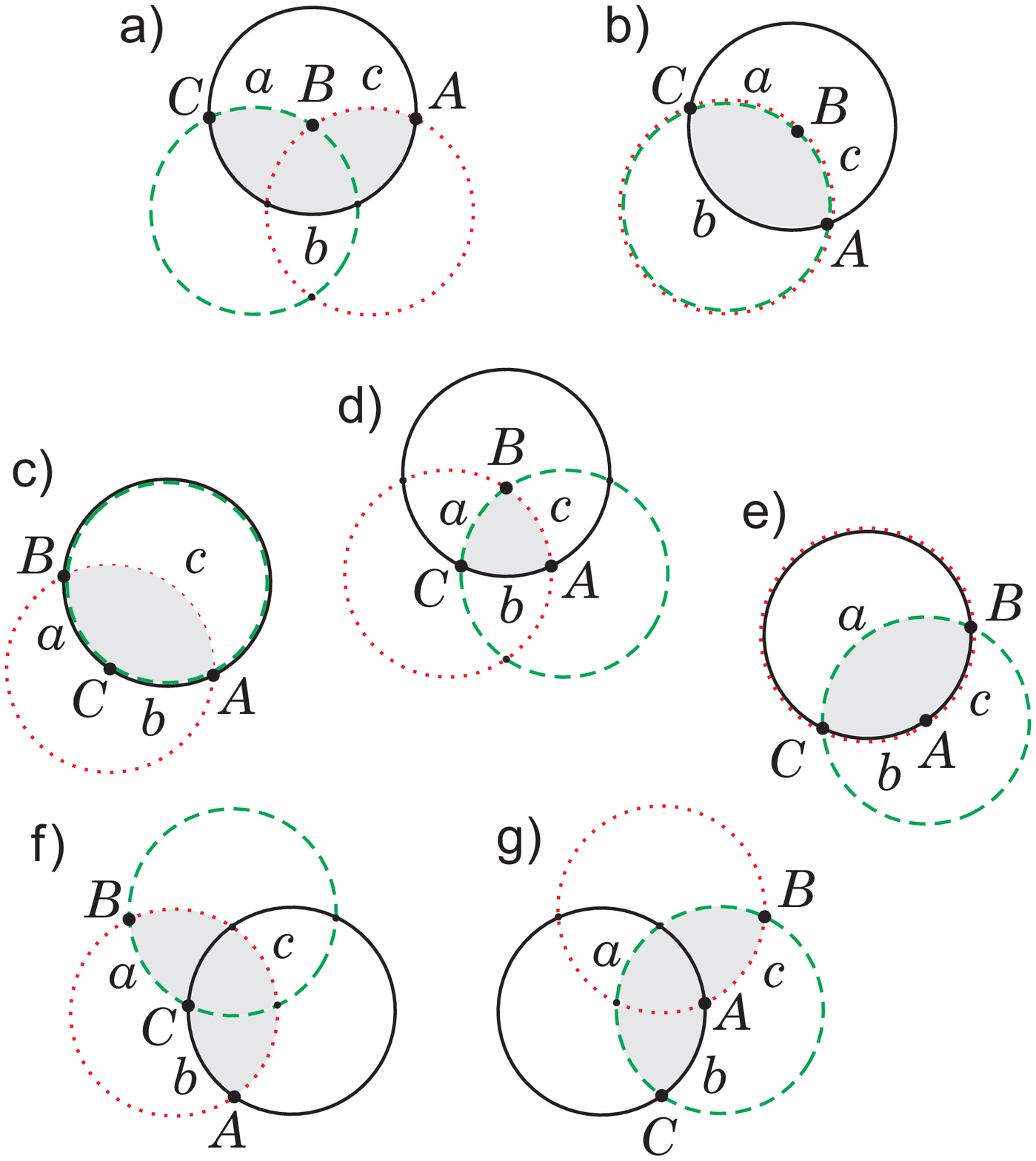}
\caption{Triangles with $A+B+C<3$ in a neighborhood $U$ of the vertex $(1,1,1)$.
When two sides of a triangle are mapped to the same circle, that circle is shown in dual color/style.}\label{down}
\end{figure}
\begin{figure}
\centering
\includegraphics[width=4.3in]{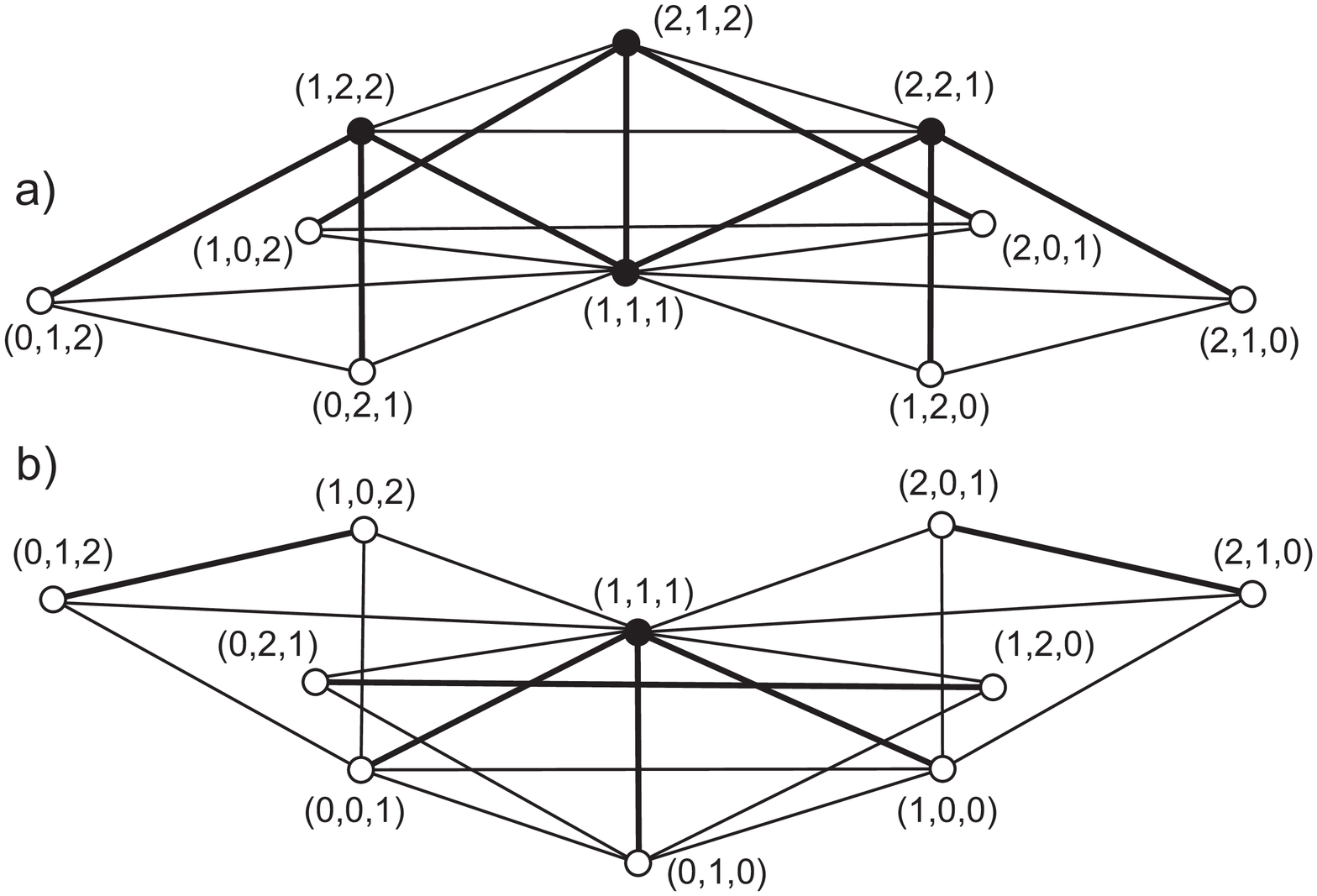}
\caption{The upward (a) and downward (b) tetrahedra adjacent to the vertex $(1,1,1)$ in Lemma \ref{111}. Edges in $X_1$ are shown in bold line.}\label{000}
\vskip.1in
\includegraphics[width=4.5in]{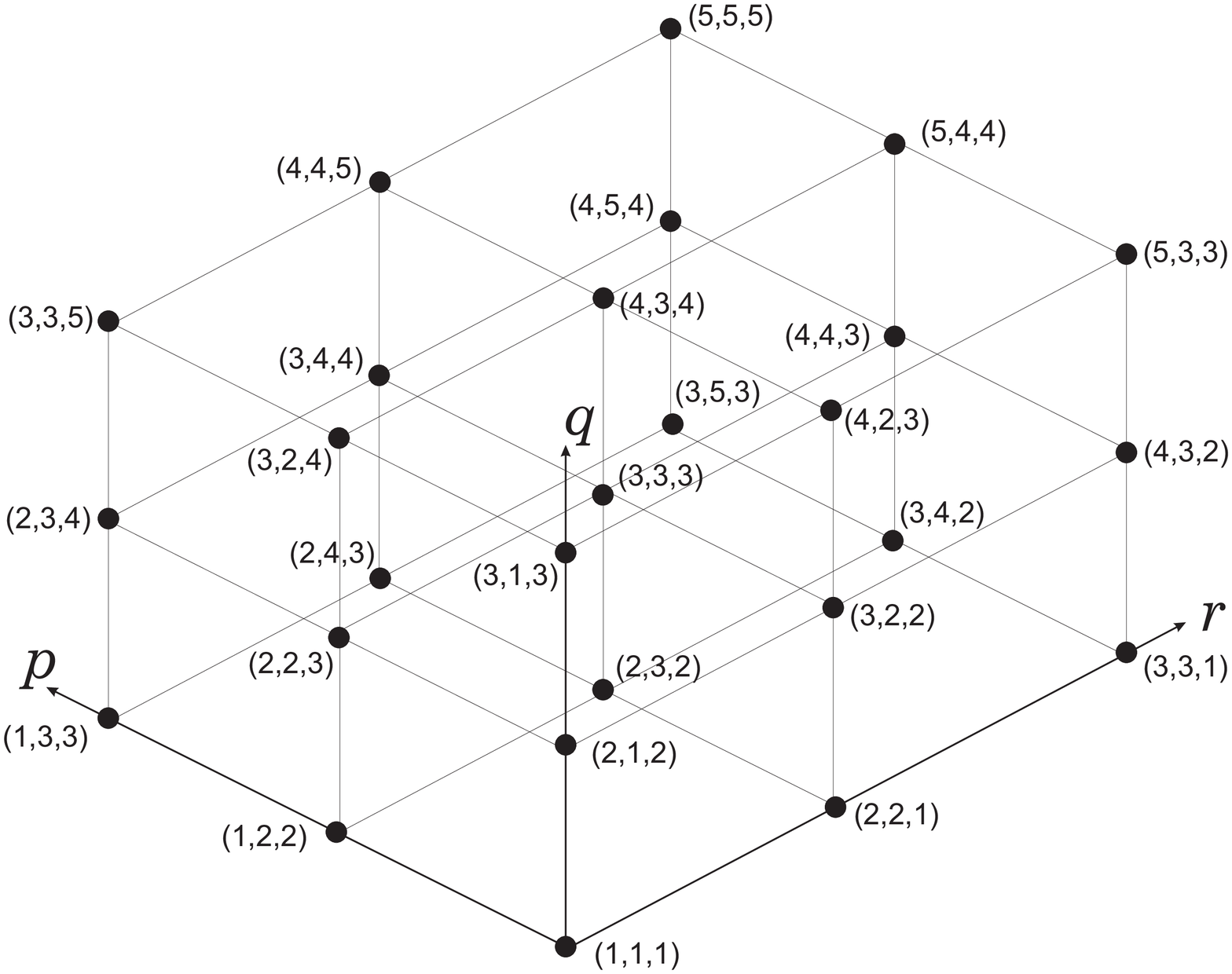}
\caption{The sets $\Lambda_0$ and $\Lambda_1$ of balanced vertices and edges in $\R^3_{p,q,r}$.}\label{cubic}
\end{figure}

In this section we define a covering of the set $\U$ of short-sided triangles (see Definition \ref{short} below) by the open neighborhoods $\U_{m,n,k}$ (see Notation \ref{umnk})
of the sets $\V_{m,n,k}$ of triangles corresponding to balanced vertices $(m,n,k)$.
Each set $\U_{m,n,k}$ is homeomorphic to an open ball in $\R^3$.

\begin{lemma}\label{111}
Let $U\subset\X$ be the union of the sets of triangles corresponding to the vertex $V=(1,1,1)$, edges in $X_1$ meeting at $V$,
and tetrahedra in $X_3$ adjacent to $V$. Then $U$ is an open neighborhood of the set $\V_{1,1,1}$ in $\X$ consisting of all triangles
with the angles $A,B,C$ satisfying the inequalities
\begin{equation}\label{abc}
0<A<2,\quad 0<B<2,\quad 0<C<2.
\end{equation}
In addition, the angles and side lengths of triangles in $U$ satisfy
\begin{equation}\label{in4}
1<A+B+C<5,\quad a<2,\quad b<2,\quad c<2,
\end{equation}
with at most one side length being $\ge 1$.
\end{lemma}

\begin{proof}
It follows from Section \ref{angles} that the union of the vertex $V$, edges in $X_1$ meeting at $V$ and tetrahedra in $X_3$ adjacent to $V$ coincides with the intersection of the set $X$ with the open cube $(0,2)^3$ defined by the inequalities (\ref{abc}).
Thus $U$ is the preimage of $X\cap (0,2)^3$ in $\X$, which is an open neighborhood in $\X$ of the preimage $\V_{1,1,1}$ of $V$.

Each triangle corresponding to $V$ is a hemisphere with three distinct marked points $(\A,\B,\C)$ at the boundary.
The set $\V_{1,1,1}$ of these triangles can be parametrized by any two of the side lengths $(a,b,c)$, since $a+b+c=2$.

There are six edges in $X_1$ meeting at $V$ (see Proposition \ref{sixedges}). Three upward edges connect $V$ with the
balanced vertices $(1,2,2)$, $(2,1,2)$ and $(2,2,1)$, and three downward edges connect it with the unbalanced vertices $(1,0,0)$, $(0,1,0)$ and $(0,0,1)$.

Triangles in $\L_{\frac32,1,\frac32}$ (see Fig.~\ref{up}b) corresponding to the upward edge from $V$ to $(2,1,2)$ have the angles $1<A=C<2,\;B=1$ and side lengths $a+c=1,\;b=1$. The set $\L_{\frac32,1,\frac32}$ can be parametrized by the angle $A$ and either $a$ or $c$.
The limits of these triangles in $\V_{1,1,1}$ have the side lengths $a+c=1,\;b=1$.
This edge is common to the tetrahedra $T_{1,0,1}$ and $T_{1,1,1}$.
Triangles corresponding to these two tetrahedra are shown in Fig.~\ref{up}a and Fig.~\ref{up}d.
The angles of triangles in $T_{1,0,1}$ satisfy the inequalities
$1<A<2$, $1<C<2$, $B<1$, $3<A+B+C<5$,
and the side lengths are $1<b<2$, $a<1$, $c<1$, $a+b+c>2$.
The limits of these triangles in $\V_{1,1,1}$  have side lengths $1\le b<2$, $a+c=2-b\le 1$.
The angles of triangles in $T_{1,1,1}$ satisfy the inequalities
$1<A<2$, $1<B<2$, $1<C<2$, $A+B+C<5$,
and the side lengths are $a<1$, $b<1$, $c<1$, $a+b+c<2$.
The limits of these triangles in $\V_{1,1,1}$  have all side lengths $\le 1$.
Both tetrahedra are pointed: $T_{1,1,1}$ has three edges in $X_1$ meeting at $V$, and $T_{1,0,1}$ has three edges in $X_1$ meeting at $(2,1,2)$ (see Fig.~\ref{000}a).

 Triangles corresponding to the edges from $V$ to $(2,2,1)$ and $(1,2,2)$ are shown in Fig.~\ref{up}c and Fig.~\ref{up}e.
 They have the angles $1<A=B<2,\;C=1$ and $1<B=C<2,\;A=1$, respectively, and side lengths satisfying $a+b=1,\;c=1$ (resp., $b+c=1,\;a=1$).
 The limits of these triangles in $\V_{1,1,1}$  have the side lengths $a+b=1,\;c=1$ (resp., $b+c=1,\;a=1$).
 These two edges are common for $T_{1,1,1}$ and the tetrahedra $T_{1,1,0}$ and $T_{0,1,1}$, respectively.
 Triangles corresponding to these two tetrahedra are shown in Fig.~\ref{up}f and Fig.~\ref{up}g.

Triangles in $\L_{\frac12,1,\frac12}$ (see Fig.~\ref{down}b) corresponding to the downward edge from $V$ to $(0,1,0)$ have the angles $0<A=C<1,\;B=1$ and side lengths $a+c=1,\;b=1$. The set $\L_{\frac12,1,\frac12}$ can be parametrized by the angle $A$ and either $a$ or $c$.
The limits of these triangles in $\V_{1,1,1}$ have the side lengths $a+c=1,\;b=1$.
This edge is common to the tetrahedra $T_{0,1,0}$ and $T_{0,0,0}$.
Triangles corresponding to these tetrahedra are shown in Fig.~\ref{down}a and Fig.~\ref{down}d.
The angles of triangles in $T_{0,1,0}$ satisfy the inequalities
$A<1$, $C<1$, $1<B<2$, $A+B+C<3$, and the side lengths are $1<b<2$, $a<1$, $c<1$, $a+b+c>2$.
The limits of these triangles in $\V_{1,1,1}$  have the side lengths $1\le b<2$, $a+c=2-b\le 1$.
All angles and side lengths of triangles in $T_{0,0,0}$ are less than 1, $A+B+C>1$, $a+b+c<2$.
The limits of these triangles in $\V_{1,1,1}$  have all side lengths $\le 1$.
The tetrahedron $T_{0,0,0}$ is pointed (with three edges in $X_1$ meeting at $V$), while the tetrahedron $T_{0,1,0}$ is not pointed (see Fig.~\ref{000}b).

 Triangles corresponding to the edges from $V$ to $(0,0,1)$ and $(1,0,0)$ are shown in Fig.~\ref{down}c and Fig.~\ref{down}e.
 They have the angles $A=B<1,\;C=1$ and $B=C<1,\;A=1$, respectively, and side lengths $a+b=1,\;c=1$ (resp., $b+c=1,\;a=1$).
 The limits of these triangles in $\V_{1,1,1}$ have the side lengths $a+b=1,\;c=1$ (resp., $b+c=1,\;a=1$).
 These two edges are common for $T_{0,0,0}$ and the tetrahedra $T_{0,0,1}$ and $T_{1,0,0}$, respectively.
 Triangles corresponding to these tetrahedra are shown in Fig.~\ref{down}f and Fig.~\ref{down}g.
\end{proof}

\begin{remark}\label{tau}{\em
The involution $\tau:(A,B,C)\mapsto(2-A,2-B,2-C)$ maps the tetrahedron $T_0=T_{0,0,0}$ to the tetrahedron $T_1=T_{1,1,1}$, and the tetrahedra $T_{1,0,0},\;T_{0,1,0}$ and $T_{0,0,1}$ to the tetrahedra $T_{0,1,1},\;T_{1,0,1}$ and $T_{1,1,0}$, respectively. For each triangle $\Delta\subset\bC$ in one of these tetrahedra, the triangle $\tau(\Delta)=\overline{\bC\setminus\Delta}$ is reflection symmetric to the complementary to $\Delta$ triangle. In particular, $\tau$ preserves the side lengths of triangles, and can be extended to the neighborhood $U$ of $\V_{1,1,1}$.}
\end{remark}


\begin{prop}\label{U}
 The neighborhood $U$ of $\V_{1,1,1}$ in Lemma \ref{111} is a three-dimensional real analytic manifold homeomorphic to an open ball in $\R^3$.
 The angles $A,B,C$ and side lengths $a,b,c$ of triangles in $U$ are real analytic functions on $U$.
 They define an embedding of $U$ to $\R^6$ as a real analytic submanifold.
\end{prop}

\begin{proof}
Let $D\subset\bC$ be a hemisphere bounded by a great circle $G$, with a marked point $\B\in G$.
For a point $\C\in G\setminus\B$, let $\B\C$ be the arc of $G$ of length $a<2$ from $\B$ to $\C$
oriented consistently with the orientation of $G=\partial D$.
Let $\B'$ and $\C'$ be the points of $G$ opposite to $\B$ and $\C$, and let $\C'\B'$ be the arc of $G$
such that either $\B\C\subseteq \C'\B'$ or $\C'\B'\subseteq \B\C$.
Let $\Gamma_\C=\B\C\cup \C'\B'$. Then $G\setminus\Gamma_\C$ is a non-empty open arc of $G$ of length $\min(a,2-a)$.
For a point $\A\in\bC\setminus\Gamma_\C$, let $\Delta_{\A\C}\subset\bC$ be a spherical triangle with the sides
$\A\B$ and $\A\C$ of lengths $c<1$ and $b<1$, respectively, and the side $\B\C\subset G$.
The set $U_A$ of triangles $\Delta_{\A\C}$ is an open subset of $U$,
parametrized by the length $a\in(0,2)$ of the arc $\B\C$ and the point $\A\in\bC\setminus\Gamma_\C$.

Let us show that there is one-to-one correspondence between triangles $\Delta_{\A\C}\in U_\A$ and spherical triangles with the angles satisfying (\ref{abc}) and the side lengths $a<2,\;b<1,\;c<1$.
Since the sides $\A\B$ and $\A\C$ of $\Delta_{\A\C}$ have lengths less than $1$, they do not intersect $G$ at any points other than $\B$ and $\C$, respectively. Since $\B\C$ is oriented from $\B$ to $\C$ in $\partial D$, this implies that $D$ is a proper subset of $\Delta_{\A\C}$ when $\A\notin D$, $\Delta_{\A\C}=D$ when $\A\in G\setminus\Gamma_\C$, and $\Delta_{\A\C}$ is a proper subset of $D$ when $\A\in D\setminus G$. Accordingly, each triangle $\Delta_{\A\C}$ may be of the following types:

\noindent(i) Triangle in $\V_{1,1,1}$ with the sides $b<1$ and $c<1$, when $\A\in G\setminus\Gamma_\C$;

\noindent(ii) Triangle in $\L_{1,\frac12,\frac12}$ corresponding to the edge $\big((1,1,1),(1,0,0)\big)$ when $\A\in D\setminus G$ and $a=1$ (see Fig.~\ref{down}e);

\noindent(iii) Triangle in $\L_{1,\frac32,\frac32}$ corresponding to the edge $\big((1,1,1),(1,2,2)\big)$, when $\A\notin D$ and $a=1$ (see Fig.~\ref{up}e).

\noindent(iv) Triangle in $T_{0,0,0}$, when $\A\in D\setminus G$ and $a<1$ (see Fig.~\ref{down}d);

\noindent(v) Triangle in $T_{1,0,0}$, when $\A\in D\setminus G$ and $a>1$ (see Fig.~\ref{down}g);

\noindent(vi) Triangle in $T_{1,1,1}$, when $\A\notin D$ and $a<1$ (see Fig.~\ref{up}d);

\noindent(vii) Triangle in $T_{0,1,1}$, when $\A\notin D$ and $a>1$ (see Fig.~\ref{up}g);

Conversely, each triangle in (i) - (vii) can be realized as a unique triangle $\Delta_{\A\C}\in U_A$, for a fixed point $\B\in G$, and for some points $\C\in G\setminus \B$ and $\A\in\bC\setminus\Gamma_\C$.

Finally, $U_A$ is projected to an open interval $(0,2)$ of the values of $a$, with the fiber $\bC\setminus\Gamma_\C$ homeomorphic to an open disk and continuously depending on $a$. Thus $U_A$ is homeomorphic to an open ball in $\R^3$.

By a cyclic permutation of the labels $(A,B,C)$, we define the sets $U_B$ and $U_C$ of triangles with the angles satisfying (\ref{abc}) and sides satisfying $b<2,\;a<1,\;c<1$ and $c<2,\;a<1,\;b<1$, respectively.
The same arguments as above show that each of these two sets is homeomorphic to an open ball in $\R^3$.
The intersection $U_A\cap U_B = U_A\cap U_C = U_B\cap U_C$ consists of triangles in $U$ with all three side lengths less than 1, corresponding to triangles in $\V_{1,1,1}$ with $\max(a,b,c)<1$, $T_{0,0,0}$ and $T_{1,1,1}$.
The set of such triangles, also homeomorphic to an open ball in $\R^3$, can be parametrized, as a subset of $U_A$, by the length $a\in(0,1)$ of the side $\B\C$ and the point $\A\in\bC\setminus\Gamma_\C$.
Since triangles in $U$ have at most one side length $\ge 1$, we have $U=U_A\cup U_B\cup U_C$.
Thus $U$ is homeomorphic to an open ball in $\R^3$.

To prove that $U$ is embedded in $\R^6$ as a real analytic manifold, we show this first for embedding of the chart $U_A$ of $U$,
parametrized by the length $a\in(0,2)$ of the side $\B\C$ and the point $\A\in\bC\setminus\Gamma_\C$.

We realize $\bC$ as the unit sphere in $\R^3_{x,y,z}$ and $G$ as the unit circle in the $xy$-plane, and set $\B=(1,0,0)$, $\C=(s,t,0)$ where $s^2+t^2=1$,
$\A=(u,v,w)$ where $u^2+v^2+w^2=1$. The opposite points of $\B$ and $\C$ in $G$ are $\B'=(-1,0,0)$ and $\C'=(-s,-t,0)$.
Then
\begin{equation}\label{gammac}
\Gamma_\C=(x,y,0)\in G: y>0\;\text{or}\;sy<tx.
\end{equation}
The sides $\A\B$, $\B\C$ and $\C\A$ belong to the circles in the planes through the origin of $\R^3$ with the normals
$(0,-w,v)$, $(0,0,1)$ and $(tw,-sw,sv-tu)$, respectively.
All three normals are non-zero vectors for any point in $U_A$, as $\A\notin\Gamma_\C$ implies that $v\ne 0,\;sv\ne tu$ when $w=0$,
and $(s,t)$ is a unit vector.
Thus all three planes depend analytically on parameters in $U_A$,
and the angles $A,B,C$ between any two of these planes are real analytic functions on $U_A$.
The side lengths $b<1$ and $c<1$ are also real analytic functions of parameters in $U_A$,
as $\cos(\pi c)=u,\;\cos(\pi b)=su+tv$.
The mapping from $U_A$ to $\R^3_{a,b,c}$ is nondegenerate when
$$ds\wedge du\wedge d(su+tv)\wedge d(s^2+t^2)\wedge d(u^2+v^2+w^2)=4t^2 w\,ds\wedge dt\wedge du\wedge dv\wedge dw\ne 0.$$
When $w=0$, all corners $\A,\B,\C$ are on the unit circle in the $xy$-plane, all angles $A,B,C$ are equal to $1$,
and the sides $a,b,c$ satisfy $a+b+c=2$.
For fixed $a$ and $b$ we have $\partial C/\partial w=-1/(\pi^2 b)\ne 0$ when $w=0$,
thus the mapping of $U_A$ to $\R^3_{a,b,C}$ is nondegenerate in this case.
When $t=0$, we have $\C=(-1,0,0)$, $a=1$, $b+c=1$, $A=1$, and $0<B=C<2$ is the angle between the vectors $(0,v,w)=\A-(u,0,0)$ and $(0,-1,0)$
counterclockwise in the $yz$-plane. Note that $u\ne\pm 1$ since $\A\ne\B$ and $\A\ne\C$.
Thus the mapping from $U_A$ to $\R^3_{a,b,C}$ is nondegenerate in this case.
This proves that $U_A$ is embedded in $\R^6$ as a real analytic manifold.

Embedding of the charts $U_B$ and $U_C$ of $U$ is obtained by cyclic permutations of the corners $\A,\B,\C$, angles $A,B,C$ and side lengths $a,b,c$.
Also, the common intersection of any two of the charts $U_A$, $U_B$ and $U_C$ of $U$ is mapped to itself by these cyclic permutations,
which act as linear transformations of $\R^6$.
Thus transition maps between the three charts of $U$ are real analytic.
\end{proof}

\begin{definition}\label{short} \emph{A spherical triangle $\Delta$ is called \emph{short-sided} if all its sides are shorter than the full circle (have length less than $2$). Otherwise, $\Delta$ is \emph{long-sided}. At most one side of a triangle $\Delta$ may be long.
If the side $\B\C$ of a triangle $\Delta$ is short, a hemisphere can be attached to $\B\C$,  increasing the angles $B$ and $C$ of $\Delta$ by $1$,
and replacing the side length $a$ of $\B\C$ by $2-a$. This operation can be repeated, attaching several hemispheres to $\B\C$. Similarly, hemispheres
can be attached to short sides $\A\C$ and $\B\C$ of $\Delta$. The triangle $\Delta(p,q,r)$ obtained by attaching $p$ hemispheres to the side $\B\C$
of a triangle $\Delta$ with angles $(A,B,C)$ and side lengths $(a,b,c)$, $q$ hemispheres to its side $\A\C$ and $r$ hemispheres to its side $\A\B$ has the angles $(A+q+r,B+p+r,C+p+q)$ and side lengths $\big((-1)^p(a-1)+1,(-1)^q(b-1)+1,(-1)^r(c-1)+1\big)$. The value $0$ for $p$, $q$ or $r$ means that no hemispheres are attached to the corresponding side of $\Delta$. If a side of $\Delta$ is long then hemispheres cannot be attached to that side, thus the corresponding value of $p$, $q$ or $r$ must be $0$.}
\end{definition}

\begin{prop}\label{x0} \emph{(See \cite{EGT1}, Section 10.)}
For a balanced vertex $(m,n,k)\in X_0$ there exists a unique solution $(p,q,r)\in\N^3$ of the system
\begin{equation}\label{pqr}
q+r=m-1,\quad p+r=n-1,\quad p+q=k-1.
 \end{equation}
 This identifies $X_0$ with the set $\Lambda_0$ of integer points in the first octant of $\R^3_{p,q,r}$, where the origin
 $(p,q,r)=(0,0,0)$ corresponds to the vertex $(m,n,k)=(1,1,1)$ (see Fig.~\ref{cubic}).
 A triangle $\Delta(p,q,r)$ with integer angles $(m,n,k)$ can be obtained from a hemisphere $\Delta$ with distinct boundary points $(\A,\B,\C)$ by attaching $p$ hemispheres to the side $\B\C$, $q$ hemispheres to the side $\A\C$, and $r$ hemispheres to the side $\A\B$.
 The developing map of $\Delta(p,q,r)$ is a rational function with three
critical points at $\A$, $\B$ and $\C$ of multiplicities $m-1$, $n-1$ and $k-1$.
\end{prop}

Let $S$ be the commutative semigroup of $G_1$ generated by translations
$(A,B,C)\mapsto (A,B+1,C+1)$, $(A,B,C)\mapsto (A+1,B,C+1)$,
$(A,B,C)\mapsto (A+1,B+1,C)$.
For each balanced vertex $W=(m,n,k)$, translation by $(m-1,n-1,k-1)$
in $S$ maps the vertex $V=(1,1,1)$ to $W$,
and the edges in $X_1$ and tetrahedra in $X_0$
adjacent to $V$ to the edges and tetrahedra adjacent to $W$.

\begin{notation}\label{umnk}
Let $\U_{m,n,k}$ be the neighbprhood of the set $\V_{m,n,k}$
consisting of triangles
with the angles $(A,B,C)$ such that $m-1<A<m+1$, $n-1<B<n+1$, $k-1<C<k+1$.
\end{notation}

Since all triangles in the neighborhood $U$ of the set $\V_{1,1,1}$
are short-sided,
Proposition \ref{x0} implies that this action
can be extended to the mapping from $U=\U_{1,1,1}$ to
$\U_{m,n,k}$
by attaching $p$ hemispheres to the side $\B\C$,
$q$ hemispheres to the side $\A\C$,
and $r$ hemispheres to the side $\A\B$ of each triangle
$\Delta\in U$, where $(p,q,r)$ satisfies (\ref{pqr}).

The sides of the resulting triangle are either of the same
length $a$, $b$, $c$ as the corresponding sides of $\Delta$
or of the complementary length $2-a$, $2-b$, $2-c$,
depending on the parity of the numbers $p$, $q$, $r$.
For example, the sides of triangles in $\U_{1,2,2}$ are $(2-a,\;b,\;c)$,
since $(p,q,r)=(1,0,0)$ in that case.
In particular, all triangles in each set $\U_{m,n,k}$ are short-sided.

\begin{theorem}\label{shortsided}
The set of all short-sided spherical triangles is an orientable three-dimensional manifold in $\R^6$ consisting of triangles corresponding to all balanced vertices in $X_0$, balanced and semi-balanced edges in $X_1$, and balanced and semi-balanced tetrahedra in $X_3$. It is the union $\U$ of the sets $\U_{m,n,k}$ corresponding to all balanced vertices $(m,n,k)$.
\end{theorem}

\begin{proof}
It will be shown in the next section that all triangles in $\X\setminus\U$
are long-sided (have one side of length $\ge 2$). Since $\U$ is a three-dimensional manifold covered by charts
$\U_{m,n,k}$, it is enough to show that $\U$ is orientable.

The set $U=U_{1,1,1}$ is the union $U=U_A\cup U_B\cup U_C$ of three open subsets sets, each of them naturally oriented as a subset of $(0,2)\times\bC$ (see proof of Proposition \ref{U}). It is easy to check (selecting the point $\A$ either at the center of $D$ or at the center of $\bC\setminus D$ in the proof of Proposition \ref{U}) that these orientations are compatible on the intersections $U_A\cap U_B = U_A\cap U_C=U_B\cap U_C$ (opposite to the orientation of $\R^3_{A,B,C}$ on the tetrahedra $T_{0,0,0}$ and $T_{1,1,1}$). Thus the set $U$ is oriented. Note that this orientation of $U$ is compatible with the
orientation of $\R^3_{A,B,C}$ on all tetrahedra adjacent to $(1,1,1)$ except $T_{0,0,0}$ and $T_{1,1,1}$.

A generator $(A,B,C)\mapsto (A,B+1,C+1)$ of $S$ maps the set $U_{1,1,1}$ to the set $U_{1,2,2}$ which intersects
with $U_{1,1,1}$ over $T_{0,1,1}\cup T_{1,1,1}\cup\L_{1,\frac32,\frac32}$. This mapping defines orientation of $U_{1,2,2}$ such that orientations of $U_{1,1,1}\cap U_{1,2,2}$ induced from $U_{1,1,1}$ and $U_{1,2,2}$ are opposite.
Orientation of the set $\U$ can be defined by reversing orientations of all sets $U_{m,n,k}$ induced from the orientation of $U_{1,1,1}$ when $m+n+k\equiv 1\mod 4$, corresponding to the odd values of $p+q+r$.
\end{proof}

\section{Sequences of unbalanced tetrahedra in $X_3$ and edges in $X_1$}\label{semibalanced}

\begin{prop}\label{unbalanced}
Let $L_0$ be a semi-balanced edge in $X_1$.
Then there is a unique not pointed tetrahedron $\nabla_0$ in $X_3$ with the edge $L_0$, and a unique infinite sequence
\begin{equation}\label{sequence}
\nabla_0,L_1,\nabla_1,L_2,\nabla_2,\ldots
\end{equation}
 of tetrahedra $\nabla_j$ in $X_3$ and edges $L_j$ in $X_1$, where $\nabla_{j-1}$ and $\nabla_j$ have a common edge $L_j$, for each $j>0$.
 The tetrahedra $\nabla_j$ and edges $L_j$ are unbalanced for $j>0$.
Each unbalanced tetrahedron in $X_3$, and each unbalanced edge in $X_1$, belongs to exactly one such sequence.
\end{prop}

\begin{proof}
First we construct the sequence (\ref{sequence}).
Let $W=(m,n,k)$ be a marginally unbalanced vertex such that $m-1=n+k$,
and let $L_0=\big((m,n+1,k+1),(m,n,k)\big)$ be a semi-balanced edge
in $X_1$ of the second type, with one end at $W$.
Then $L_0$ is a common edge of a pointed semi-balanced tetrahedron
$T_{m-1,n,k}$ (with two more edges in $X_1$,
from $(m,n+1,k+1)$ to $(m-1,n,k+1)$ and $(m-1,n+1,k)$)
and a non-pointed semi-balanced tetrahedron $\nabla_0=T_{m,n,k}$
with one unbalanced edge $L_1=\big((m+1,n,k+1),(m+1,n+1,k)\big)$ in $X_1$,
of the first type.
The edge $L_1$ is common for the tetrahedron $\nabla_0$
and an unbalanced
tetrahedron $\nabla_1=T_{m+1,n,k}$ which
has an unbalanced edge $L_2=\big((m+2,n+1,k+1),(m+2,n,k)\big)$
in $X_1$, of the second type.
Extending this construction, we obtain the sequence (\ref{sequence})
consisting of unbalanced edges $L_j=\big((m+j,n,k+1),(m+j,n+1,k)\big)$
of the first type for odd $j$, unbalanced edges
$L_j=\big((m+j,n+1,k+1),(m+j,n,k)\big)$ of the second type for even $j$,
and the tetrahedra $\nabla_j=T_{m+j,n,k}$, unbalanced for $j>0$.

The cases $n-1=m+k$ and $k-1=m+n$ are similar.

Now we show that each unbalanced tetrahedron in $X_3$ and each unbalanced
edge in $X_1$ belong to exactly one sequence (\ref{sequence}).
According to Remark \ref{edges-unbalanced} each unbalanced edge in $X_1$
is an edge of an unbalanced tetrahedron,
and each unbalanced tetrahedron $\nabla=T_{m,n,k}$ in $X_3$ with $m>n+k+1$
has two opposite unbalanced edges in $X_1$,
either $L=\big((m,n,k+1),(m,n+1,k)\big)$ and
$L'=\big((m+1,n+1,k+1),(m+1,n,k)\big)$
or $L=\big((m+1,n,k+1),(m+1,n+1,k)\big)$
and $L'=\big((m,n+1,k+1),(m,n,k)\big)$,
depending on the parity of $m+n+k$.
Then $\nabla$, $L$ and $L'$ belong to a sequence
(\ref{sequence}) associated with a semi-balanced edge
$L_0=\big((n+k+1,n+1,k+1),(n+k+1,n,k)\big)$.
The cases $n>m+k+1$ and $k>m+n+1$ are similar.
\end{proof}

\begin{prop}\label{sausage}
The set $\S\subset\X$ of triangles corresponding to the tetrahedra $\nabla_j$ for $j\ge 0$ and edges $L_j$ for $j>0$ in a sequence (\ref{sequence}) is homeomorphic to an open ball in $\R^3$.
All triangles in $\S$ except those in $\nabla_0$ are long-sided.
\end{prop}

\begin{proof}
 We start with the case $W=(1,1,1)$ and $L_0=\big((1,1,1),(1,0,0)\big)$. Triangles corresponding to the edges $L_j$ and tetrahedra $\nabla_j$ of the sequence (\ref{sequence}) can be constructed as follows. Let $D\subset\bC$ be a hemisphere bounded by a great circle $G$. For a fixed point $\B\in G$, let $\C\in G$ be an opposite point, and let $\B\C$ be an arc of $G$ oriented from $\B$ to $\C$ consistently with the orientation of $G=\partial D$.
Let us choose an arbitrary point $\A\in D\setminus G$, and let $G_1$ be the great circle
passing through $\A$ and $\B$ (and also through $\C$, since $\C$ is opposite to $\B$).
Connecting $\A$ with $\B$ and $\C$ by the arcs of $G_1$ inside $D$, we get a triangle $\delta_{\A,1}\subset D$ with the vertices $(\A,\B,\C)$, angles $A=1$, $B=C<1$ and sides $a=1$, $b+c=1$ (as in Fig.~\ref{down}e) corresponding to the edge $L_0$. Since the point $\A$ is uniquely determined by the angle $B<1$ between $G$ and $G_1$ and the length $c<1$ of the arc $\A\B$, each triangle corresponding to the edge $L_0$ is equal to exactly one triangle $\delta_{\A,1}$.

 Next we fix a point $\A\in D\setminus G$, and allow the point $\C$ to move along the circle $G$, increasing the angle $A$ between the arc $\A\B$ and the arc $\A\C\subset D$ of a great circle $G_2$ passing through $\A$ and $\C$, so that $1<A<2$. Then we get triangles $\delta_{\A,A}$ with the vertices $(\A,\B,\C)$, angles $1<A<2$, $B<1$, $C<1$, and sides $1<a<2$, $b<1$, $c<1$ (as in Fig.~\ref{down}g). Each of these triangles belongs to the tetrahedron $\nabla_0=T_{1,0,0}$. Conversely, for a fixed point $\B\in G$, the point $\A\in D\setminus G$ is uniquely determined by the angles $(A,B,C)\in T_{1,0,0}$. Thus each triangle in $T_{1,0,0}$ is equal to exactly one triangle $\delta_{\A,A}$ with $1<A<2$, for some $\A\in D\setminus G$.

If we continue moving the point $\C$ along $G$ increasing the angle $A$, we get $\C=\B$ when $A=2$, and a triangle $\delta_{\A,2}$ with the angles $A=2$, $B+C=1$ and sides $a=2$, $b=c<1$ corresponding to an unbalanced edge $L_1=\big((2,0,1),(2,1,0)\big)$ of the first type. It is easy to check that each triangle corresponding to $L_1$ is equal to exactly one triangle $\delta_{\A,2}$ for some $\A\in D\setminus G$. In particular, all triangles corresponding to $L_1$ are long-sided.

Moving the point $\C$ further along $G$ and increasing the angle $A$ accordingly, we obtain a family of triangles $\delta_{\A,A}$ with vertices $(\A,\B,\C)$, angles $A\ge 1$, $B<1$, $C<1$ and sides $a\ge 1$, $b<1$, $c<1$ corresponding to all edges $L_j$ and tetrahedra $\nabla_j$ of the sequence (\ref{sequence}) associated with the semi-balanced edge $L_0$. For $j>0$ these triangles are long-sided, with $a=j+1$ for triangles corresponding to $L_j$ and $j+1<a<j+2$ for triangles in $\nabla_j$.

 For a balanced vertex $(m,n,k)$ with $m+1=n+k$ and a semi-balanced edge $L_0=\big((m,n,k),(m,n-1,k-1)\big)$, we fix a point $\B\in G=\partial D$ and define a triangle $\delta_{\A,A,n,k}$, for some point $\A\in D\setminus G$ and angle $A\ge 1$, as a triangle $\delta_{\A,A}$ with $n-1$ hemispheres attached to its side $\A\B$ and $k-1$ hemispheres attached to its side $\A\C$.
Then $\delta_{\A,A,n,k}$ has the angles $(A+m-1,B,C)$ where $A+m-1\ge m$, $n-1<B<n$, $k-1<C<k$,
and sides $a\ge 1$, $b<1$ if $k$ is odd, $1<b<2$ if $k$ is even, $c<1$ if $n$ is odd, $1<c<2$ if $n$ is even.
It is easy to check that these triangles $\delta_{\A,A,n,k}$ are in one-to-one correspondence with triangles corresponding to the edges $L_j$ and tetrahedra $\nabla_j$ of the sequence (\ref{sequence}) in Proposition \ref{unbalanced} associated with the semi-balanced edge $L_0=\big((m,n,k),(m,n-1,k-1)\big)$.
For $j>0$ all these triangles are long-sided.

This construction identifies the set $\S$ with the product $(D\setminus G)\times(1,\infty)$ homeomorphic to an open ball in $\R^3$.

The cases of balanced vertices $(m,n,k)$ with $n+1=m+k$ and $k+1=m+n$, and semi-balanced edges
$L_0=\big((m,n,k),(m-1,n,k-1)\big)$ and $L_0=\big((m,n,k),(m-1,n-1,k)\big)$, are similar.
\end{proof}

\section{Homotopy type of the set $\X$ of spherical triangles}\label{type}

According to Proposition \ref{x0}, the set $X_0$ of balanced vertices $(m,n,k)$ can be identified with the set $\Lambda_0$ of the integer points $(p,q,r)\in\N^3$ in the first octant of $\R^3_{p,q,r}$ so that $(m,n,k)=(q+r+1,p+r+1,p+q+1)$.
It follows from Proposition \ref{sixedges} that an edge in $X_1$ with one end at a balanced vertex $(m,n,k)$ has its other end either at an unbalanced vertex or at a balanced vertex identified with $(p',q',r')=(p\pm 1,q\pm 1,r\pm 1)\in\N^3$.
Thus the union of $X_0$ and the set $B_1$ of balanced edges in $X_1$ can be identified with the 1-skeleton $\Lambda_1$ of the cubic partition of the first octant of $\R^3_{p,q,r}$ (see Fig.~\ref{cubic}).

\begin{theorem}\label{homotopy}
The set $\X$ of all spherical triangles is homotopy equivalent to $\Lambda_1$.
\end{theorem}

\begin{proof}
According to the proof of Proposition \ref{sausage}, the set $\S$ of triangles in the sequence (\ref{sequence}) associated with a semi-balanced edge $L_0=\big((m,n,k),(m,n-1,k-1)\big)$ can be parametrized by $(\A,A)\in(D\setminus\partial D)\times(1,\infty)$ where $\A$ is a point in the interior of a hemisphere $D$ and $A$ is the angle at $\A$. The union of $\S$ and the set $\L_0$ of triangles corresponding to $L_0$
can be parametrized by $(\A,A)\in(D\setminus\partial D)\times[1,\infty)$, and the union of $\L_0$ and the set of triangles in the semi-balanced tetrahedron $\nabla_0$ can be parameterized by $(\A,A)\in(D\setminus\partial D)\times[1,2)$.
A homeomorphism $[1,\infty)\to[1,2)$ defines a homeomorphism $\L_0\cup\S\to\L_0\cup\nabla_0$.
Applying this homeomorphism to the sequences (\ref{sequence}) corresponding to all semi-balanced edges in $X_1$, we obtain a homeomorphism
$\X\to\U$ where $\U$ is the set of all short-sided triangles (see Definition \ref{short} and Theorem \ref{shortsided}) corresponding to all balanced vertices in $X_0$, balanced and semi-balanced edges in $X_1$, and balanced and semi-balanced tetrahedra in $X_3$. This homeomorphism is identity on the triangles corresponding to all vertices in $X_0$ and all balanced and semi-balanced edges in $X_1$.

Next, the union of the set of triangles in each tetrahedron $\nabla\subset\U$, the sets of triangles corresponding to the edges of $\nabla$ in $X_1$ adjacent to its balanced vertices, and the sets of triangles corresponding to the balanced vertices of $\nabla$, is retractable to the union of the sets of triangles corresponding to the balanced vertices and balanced edges of $\nabla$. Moreover, this retraction can be made compatible on the common edges of the tetrahedra in $\U$, thus the set $\U$ is retractable to the union $\X_B$ of the sets of triangles corresponding to all balanced edges and vertices.
Since projection of the set $\X_B$ to $X_0\cup B_1$ is a compact-covering map with contractible fibers, Vietoris-Begle mapping theorem (see \cite{Spanier}) implies that $\X_B$ is homotopy equivalent to $X_0\cup B_1$, which can be identified with $\Lambda_1$.

Combining the homeomorphism $\X\to\U$, retraction of $\U$ to $\X_B$ and projection of $\X_B$ to $X_0\cup B_1$, we complete the proof of Theorem \ref{homotopy}.
\end{proof}

\begin{prop}\label{balanced-mnk}
If a semi-balanced tetrahedron $T$ is pointed then its subset of balanced triangles
is the intersection of $T$ with the convex hull of its balanced vertices and midpoints of all its edges
(midpoints of edges of $T$ with both ends at balanced vertices may be excluded, as they are not vertices of the convex hull).
If $T$ is not pointed then its subset of balanced triangles is the intersection of $T$
with the convex hull of its single balanced vertex $V$ and midpoints of all its edges adjacent to $V$.
\end{prop}

\begin{proof}
The set of balanced triangles in $T$ is $T_B=T\cap\K$, where $\K$ is the closed cone in the first octant of $\R^3$ defined by the inequalities (\ref{balanced}).
If $T$ is pointed then it has three edges in $X_1$ with a common end at a balanced vertex $V=(m,n,k)$
and other ends at $P=(m,n-1,k-1)$, $Q=(m-1,n,k-1)$ and $R=(m-1,n-1,k)$.
We may assume that $m+1=n+k$.
Then the vertex $P$ is unbalanced, thus the edge $VP$ is semi-balanced. Its midpoint is $vp=(m,n-\frac12,k-\frac12)$.

Consider first the case when both vertices $Q$ and $R$ are balanced. Then $T$ has two balanced edges $VQ$ and $VR$, an edge $QR$ not in $X_1$ with both ends at
balanced vertices, and two edges $PQ$ and $PR$ not in $X_1$, with
midpoints $pq=(m-\frac12,n-\frac12,k-1)$ and $pr=(m-\frac12,n-1,k-\frac12)$.
All three midpoints $vp,pq,pr$ belong to the face $A=B+C$ of $\K$, thus $T_B$ is the intersection of $T$ with the the convex hull of $V,Q,R,vp,pq,pr$.

Next, consider the case when two vertices of $T$, say $P$ and $Q$, are unbalanced, and $R$ is balanced. This happens when $m=n>k=1$.
Then $T$ has two semi-balanced edges $VP$ and $VQ$, with midpoints $vp=(m,n-\frac12,\frac12)$ and $vq=(m-\frac12,n,\frac12)$,
and three edges $PQ$, $PR$ and $QR$ not in $X_1$, with midpoints $pq=(m-\frac12,n-\frac12,0)$, $pr=(m-\frac12,n-1,\frac12)$ and $qr=(m-1,n-\frac12,\frac12)$.
Since $m=n$, the points $vp$, $pq$ and $pr$ belong to the facet $A=B+C$ of $\K$, and the points $vq$, $pq$ and $qr$ belong to the facet $B=A+C$ of $\K$.
Thus $T_B$ is the intersection of $T$ with the convex hull of $V,R,vp,vq,pq,pr,qr$.

The remaining case  $T=T_{0,0,0},\;(m,n,k)=(1,1,1)$ is left as an exercise. The set of balanced triangles in that case is the intersection of $T$ with
the convex hull of $(1,1,1)$, $(1,\frac12,\frac12)$, $(\frac12,1,\frac12)$, $(\frac12,\frac12,1)$, $(0,\frac12,\frac12)$, $(\frac12,0,\frac12)$, $(\frac12,\frac12,0)$.

If $T$ is not pointed then there is a single balanced vertex $V=(m,n,k)$ of $T$.
We may assume that $m+1=n+k$ and that a single semi-balanced edge of $T$ connects $V$ with an unbalanced vertex $P=(m,n-1,k-1)$.
Then $T$ has two more unbalanced vertices $Q=(m+1,n,k-1)$ and $R=(m+1,n-1,k)$.
The midpoints of the edges $VP$, $VQ$ and $VR$ are $vp=(m,n-\frac12,k-\frac12)$, $vq=(m+\frac12,n,k-\frac12)$ and $vr=(m+\frac12,n-\frac12,k)$.
All of them belong to the facet $A=B+C$ of $\K$. Thus $T_B$ is the intersection of $T$ with the convex hull of $V,vp,vq,vr$.
\end{proof}

\begin{cor}\label{homotopy-balanced}
The set of balanced triangles is homotopy equivalent to $\Lambda_1$.
\end{cor}

\begin{proof}
Description of the sets of balanced triangles in semi-balanced tetrahedra in Proposition \ref{balanced-mnk}
allows one to define retraction of the set $\U$ of short-sided triangles to the set $\X_B$ of triangles corresponding to balanced edges and vertices
in the proof of Theorem \ref{homotopy} as a composition of retraction of $\U$ to the set of balanced triangles and retraction of the set of balanced triangles to $\X_B$.
\end{proof}

{\em Department of Mathematics,

Purdue University, West Lafayette, IN 47907 USA

eremenko@math.purdue.edu, agabriel@math.purdue.edu}
\end{document}